\documentclass{article}
\usepackage{amsfonts}
\usepackage{amsmath}

\newtheorem{theorem}{Theorem}

\newtheorem{corollary}[theorem]{Corollary}

\newtheorem{definition}[theorem]{Definition}

\newtheorem{proposition}[theorem]{Proposition}

\newenvironment{proof}[1][Proof]{\noindent\textbf{#1.} }{\ \rule{0.5em}{0.5em}}

\begin{document}

\title{Reflected backward doubly stochastic differential equations with
time delayed generators}
\author{B. Mansouri$^{1}$, I. Salhi$^{2}$, L. Tamer$^{3}$\\
$^{1,3}$University of Biskra, Algeria.\\$^{2}$Faculty of Sciences Tunisia, University Tunis El Manar.}
\maketitle
\renewcommand{\thefootnote}{
\fnsymbol{footnote}} \footnotetext{{\scriptsize E-mail addresses:
$^{1}$mansouri.badreddine@gmail.com, $^{2}$imensalhi@hotmail.fr, $^{3}$tamerlazhar@yahoo.fr}}
\begin{abstract}
We consider a class of reflected backward doubly stochastic differential equations with time delayed generator (in short RBDSDE with time delayed generator), in this case generator at time $t$ can depend on the values of a solution in the past. Under a Lipschitz condition, we ensure the existence and uniqueness of the solution.
\end{abstract}

\section{Introduction\label{s1}}
After the earlier work of Pardoux \& Peng (1990)\cite{PP}, the theory of
backward stochastic differential equations (BSDEs in short) has a
significant headway thanks to the many application areas. Several
authors contributed in weakening the Lipschitz assumption required
on the drift of the equation (see Lepaltier \& San martin (1996)\cite{LEP},
Kobylanski (1997)\cite{KOB}, Mao (1995)\cite{MAO}, Bahlali (2000)\cite{BAH}).

A new kind of backward stochastic differential equations was
introduced by Pardoux \& Peng \cite{pp} (1994),
\begin{equation*}
Y_{t}=\xi
+\int_{t}^{T}f(s,Y_{s},Z_{s})ds+\int_{t}^{T}g(s,Y_{s},Z_{s})dB_{s}-%
\int_{t}^{T}Z_{s}dW_{s},\quad 0\leq t\leq T
\end{equation*}%
with two different directions of stochastic integrals, i.e., the
equation involves both a standard (forward) stochastic integral
$dW_{t}$ and a backward stochastic integral $dB_{t}.$ They have
proved the existence and uniqueness of solutions for backward doubly stochastic differential equations under
uniformly Lipschitz conditions. Shi et al \cite{shi}(2005)
provided a comparison theorem which is very important in studying
viscosity solution of SPDEs with stochastic tools. Bahlali et al \cite{BGM} provided the existence and uniqueness in the case with a superlinear growth generator and a square integrable
 terminal datum.

Bahlali et al \cite{bhmm} (2009) proved the existence and uniqueness
of the solution to the following reflected backward doubly
stochastic differential equations (RBDSDEs in short) with one continuous
barrier and uniformly Lipschitz coefficients:
\begin{equation}
Y_{t}=\xi +\int_{t}^{T}f\left( s,Y_{s},Z_{s}\right)
ds+\int_{t}^{T}g\left( s,Y_{s},Z_{s}\right)
dB_{s}+K_{T}-K_{t}-\int_{t}^{T}Z_{s}dW_{s},\ \ \ \ 0\leq t\leq T
\label{RBDSDE}
\end{equation}
Recently, LuO, Zhang and Li \cite{LUO} introduced the following BDSDE
\begin{equation*}
Y_{t}=\xi
+\int_{t}^{T}f(s,Y_{s},Z_{s})ds+\int_{t}^{T}g(s,Y_{s},Z_{s})dB_{s}-%
\int_{t}^{T}Z_{s}dW_{s},\quad 0\leq t\leq T
\end{equation*}
where the generator $f$ at time $s$ depends arbitrary on the past
values of a solution $(Y_s, Z_s) = (Y(s + u), Z(s + u)), T\leq u\leq0$. This
class of BDSDEs is called as BDSDEs with time delayed generators. In
Luo et al\cite{LUO}, the authors proved the existence and uniqueness
of a solution for the above equation.

In this paper, we establish the existence and uniqueness of the
solutions for RBDSDEs with time delayed generators.   The paper is
organized as follows. In Section 2, we propose some preliminaries
and notations. The section 3 we gives a priori estimates of the solution. Finally the section 4 our main result is stated and proved.
\section{Notations, assumptions and Definitions}

\hspace{0.5cm}Let $\left( \Omega,\mathcal{F},P\right) $ be a complete probability space,
and $T>0$. Let $\left\{ W_{t},0\leq t\leq T\right\} $ and $\left\{
B_{t},0\leq t\leq T\right\} $ be two independent standard Brownian motions
defined on $\left( \Omega,\mathcal{F},P\right) $ with values in $\mathbb{R}$ and $\mathbb{R}$, respectively. For $t\in\left[ 0,T\right] $, we put,
\begin{equation*}
\mathcal{F}_{t}:=\mathcal{F}_{t}^{W}\vee\mathcal{F}_{t,T}^{B}, \quad
\hbox{and} \quad \mathcal{G}_{t}:=\mathcal{F}_{t}^{W}\vee\mathcal{F}_{T}^{B}
\end{equation*}
where $\mathcal{F}_{t}^{W}=\sigma(W_{s};0\leq s\leq t)$ and $\mathcal{F}
_{t,T}^{B}=\sigma(B_{s}-B_{t};t\leq s\leq T),$ completed with $P$-null sets.
It should be noted that $\left( \mathcal{F}_{t}\right) $ is not an
increasing family of sub $\sigma-$fields, and hence it is not a filtration.
However $(\mathcal{G}_{t})$ is a filtration.

\vspace{0.5cm}
Let $M_{T}^{2}\left( 0,T,\mathbb{R}\right) $ denote the set of jointly measurable stochastic processes $\left\{
\varphi_{t};t\in\left[ 0,T\right] \right\} $, which satisfy :

\textbf{(a)} \ $E\int_{0}^{T}\left\vert \varphi_{t}\right\vert
^{2}dt<\infty. $

\textbf{(b) }\ $\varphi_{t}$ is $\mathcal{F}_{t}-$measurable, for any $t\in%
\left[ 0,T\right] .$

\vspace{0.5cm}
We denote by $S_{T}^{2}\left( \left[ 0,T\right] ,\mathbb{R}\right) ,$ the
set of continuous stochastic processes $\varphi_{t}$, which satisfy :

\textbf{(a')}\ $E\left( \sup_{0\leq t\leq T}\left\vert
\varphi_{t}\right\vert ^{2}\right) <\infty.$

\textbf{(b')} For every $t\in \left[ 0,T\right] ,\ \varphi _{t}$ is $%
\mathcal{F}_{t}-$measurable.\newline

We denote by $L_{-T}^{2}(\mathbb{R})$ the set of measurable function $\varphi:[-T,0]\rightarrow\mathbb{R}$, which satisfy : $\int_{-T}^{0}|\varphi(t)|^{2}dt<\infty.$

\vspace{0.5cm}
We denote by $S_{-T}^{\infty}(\mathbb{R})$ the set of measurable function $\varphi(t):[-T,0]\rightarrow\mathbb{R}$, which satisfy : $sup_{t\in[-T,0]}|\varphi(t)|^{2}dt<\infty.$

\vspace{0.5cm}
We denote by $L_{T}^{2}(\mathbb{R})$ the set of $\mathfrak{F}_T$-measurable random $\eta$ $\mathbb{R}$-valued, which satisfy : $E|\eta|^2<\infty$.

\vspace{0.5cm}
The spaces $M_{T}^{2}(\mathbb{R})$ and $S_{T}^{2}(\mathbb{R})$ are respectively endowed with the norms
\begin{align*}
\|\varphi\|_{M_{T}^{2}}^{2}=E\int_0^T e^{\beta t}|\varphi(t)|^{2} dt \ \  \texttt{and} \ \ \|\varphi\|_{S_{T}^{2}}^{2}=E(sup_{t\in[0,T]} e^{\beta t}|\varphi(t)|^{2}).
\end{align*}
where $\beta>0$.

\vspace{0.2cm}
We denote by $\lambda$ the Lebesgue measure on $([-T , 0], \mathbb{B}([-T , 0]))$, where $\mathbb{B}([-T , 0])$ is the Borel sets of $[-T , 0]$.

\vspace{0.5cm}We consider the following assumptions,\newline
\textbf{H1)} Let $\ f:\Omega \times \left[ 0,T\right] \times S_{-T}^{\infty}(\mathbb{R})
\times L_{-T}^{2}(\mathbb{R})\mathbb{
\rightarrow R}$ and $g:\Omega \times \left[ 0,T\right] \times S_{-T}^{\infty}(\mathbb{R})
\times L_{-T}^{2}\mathbb{
\rightarrow R}$ be two measurable functions and such that for every $\left( y,z\right) \in \mathbb{R}\times \mathbb{R}$, $f(.,y,z)$ and, $g(.,y,z)$ belongts $M^{2}(0,T,\mathbb{R})$
\newline
\textbf{H2)} There exist constants $L>0$ and $0<\alpha <1,$ such
that \ for every $\left( t,\omega \right) \in \Omega \times \left[ 0,T\right]
$ and $\left( y,z\right) \in S_{-T}^{\infty}(\mathbb{R})
\times L_{-T}^{2},$ and for a probability measure $\gamma$ on $([-T,0],\mathbb{B}([-T,0]))$
\begin{align*}
&\left\vert f( t,y_{t}^{1},z_{t}^{1})
-f( t,y_{t}^{2},z_{t}^{2}) \right\vert^{2}\\&
\leq L(\int_{-T}^{0} \left\vert y^{1}(t+u)-y^{2}(t+u)\right\vert^{2}\gamma(du) +\int_{-T}^{0}\left\vert z^{1}(t+u)-z^{2}(t+u)\right\vert^{2}\gamma(du) ) \\&
\left\vert g( t,y_{1},z_{1})
-g( t,y_{2},z_{2}) \right\vert
^{2}\leq \\& L\int_{-T}^{0} \left\vert y^{1}(t+u)-y^{2}(t+u)\right\vert^{2}\gamma(du) +\alpha\int_{-T}^{0}\left\vert z^{1}(t+u)-z^{2}(t+u)\right\vert^{2}\gamma(du)
\end{align*}
\textbf{H3)} $E\int_0^T |f(t,0,0)|^2dt<\infty,$ and $E\int_0^T |g(t,0,0)|^2dt<\infty,$\newline
\textbf{H4)} $f(t,.,.)=0$ and $g(t,.,.)=0$ for $t<0.$\newline
\textbf{H5)} Let $\xi $ be a square integrable random variable which is $%
\mathcal{F}_{T}-$mesurable. \newline
\textbf{H6)} The obstacle $\left\{ S_{t},0\leq t\leq T\right\} $, is a
continuous $\mathcal{F}_{t}-$progressively measurable real-valued process
satisfying $E\left( \sup_{0\leq t\leq T} e^{\beta t}\left( S_{t}^{+}\right) ^{2}\right)
<\infty $, for $\beta>0$. \newline
We assume also that $S_{T}\leq \xi \ a.s.$

\vspace{0.4cm}We define now RBDSDE with time delayed generators.
\begin{definition}
A solution of RBDSDE with time delayed generators is a $\left( \mathbb{R}\times \mathbb{
R}\times \mathbb{R}_{+}\right) $ -valued $\mathcal{F}_{t}-$progressively
measurable process $\left( Y_t,Z_t,K_t\right) _{0\leq t\leq T}$ which
satisfies : \newline
i) $(Y,Z,K_{T})\in S^{2}_{T}\times M^{2}_{T}\times L^{2}(\mathbb{R} )$. \newline
ii) $Y_t=\xi+\int_t^T f(s,Y_s,Z_s)ds+K_T-K_t+\int_t^T g(s,Y_s,Z_s)dB_s-\int_t^T Z_sdW_s,\ \ \ 0\leq t\leq T.$ \newline
iii)$Y_t\geq S_t,\ \ 0\leq t\leq T.$\newline
iv) $K_t $ is adapted, continuous and nondecreasing, $K_0=0$ and $
\int_{0}^{T} e^{\beta t}\left( Y_t-S_t\right) dK_t=0,$ for $\beta>0$.
\end{definition}

Note that for $(Y , Z)\in S^{2}(\mathbb{R})\times M^{2}(\mathbb{R})$ the generator $f,\ g$ are well-defined and $P-a.s$. integrable as
\begin{align}
\int_0^T |f(t,Y_t,Z_t)|^2 dt&\leq 2\int_0^T |f(t,0,0)|^2 dt+2L\int_0^T\int_{-T}^0|Y(t+u)|^2\gamma(du)dt\nonumber\\&+2L\int_0^T\int_{-T}^0|Z(t+u)|^2\gamma(du)dt\nonumber\\&
=2\int_0^T |f(t,0,0)|^2 dt+2L\int_{-T}^0\int_{u}^{T+u}|Y(v)|^2dv\gamma(du)
\nonumber\\&+2L\int_{-T}^0\int_{u}^{T+u}|Z(v)|^2dv\gamma(du)dt\nonumber\\&
\leq2\int_0^T |f(t,0,0)|^2 dt+2L(T\sup_{0\leq t\leq T} |Y(t)|^2+\int_0^T |Z(t)|^2 dt) \nonumber\\&<\infty. \label{eq:f}
\end{align}
with the same argument we get
\begin{align}
\int_0^T |g(t,Y_t,Z_t)|^2 dt<\infty. \label{eq:g}
\end{align}

\section{Priori estimates of the solution }
Using standard arguments of RBDSDEs one can prove the following estimates
\begin{proposition}
Let $(Y_t, Z_t, K_t, 0\leq t\leq T )$ be a solution of the RBDSDE with time delayed generator. If the Lipschitz constant $L$ of the generator $f$ and $g$ is small enough, then there exist two positive constants $\beta$ and $\theta$ satisfying that
$$D_1 :=\beta-\theta>0,\ \  D_2 := (1-2\widetilde{\gamma}(\frac{L}{\theta}+\alpha))>0,$$
and a positive constant $C = C(\beta,\theta, \widetilde{\gamma}, L, T, \alpha)$, ($\gamma$ as in H2) and $\widetilde{\gamma}=\int_{-T}^{0}e^{-\beta u}\gamma(du)),$ such that
\begin{align*}
&E(\sup_{0\leq t\leq T}e^{\beta t}|Y_t|^2+\int_0^T e^{\beta t}|Z_t|^2dt+e^{\beta T}|K_T|^2)\\&\leq
CE\bigg(e^{\beta T}|\xi|^2+\int_0^T e^{\beta t}|f(t,0,0)|^2dt+\int_0^T e^{\beta t}|g(t,0,0)|^2dt+\sup_{0\leq t\leq T}e^{\beta t}(S_t^+)^2\bigg).
\end{align*}
\end{proposition}
\begin{proof}
Applying the Ito formula to $e^{\beta t}|Y_t|^2$ yields that
\begin{align*}
&e^{\beta t}|Y_t|^2+\int_t^T \beta e^{\beta s}|Y_s|^2 ds
\nonumber\\& \leq e^{\beta T}|\xi|^2+\theta\int_t^T e^{\beta
s}|Y_s|^2ds+\frac{1}{\theta}\int_t^T e^{\beta
s}|f(s,Y_s,Z_s)|^2ds+2\int_t^T e^{\beta s}Y_s dK_s \nonumber\\&
+2\int_t^T e^{\beta s}Y_s g(s,Y_s,Z_s)dB_s-2\int_t^T e^{\beta s}Y_s
Z_s dW_s\nonumber\\& +\int_t^T e^{\beta
s}|g(s,Y_s,Z_s)|^2ds-\int_t^T e^{\beta s}|Z_s|^2ds
\end{align*}
by equation (\ref{eq:f}) and (\ref{eq:g}) we get
\begin{align}
&e^{\beta t}|Y_t|^2+(\beta-\theta)\int_t^T  e^{\beta s}|Y_s|^2 ds+\int_t^T e^{\beta s}|Z_s|^2ds \nonumber\\&
\leq e^{\beta T}|\xi|^2+(\frac{2LT\widetilde{\alpha}+2LT\widetilde{\alpha}}{\theta})sup_{0\leq t\leq T}e^{\beta t}|Y_{t}|^2\nonumber\\&+2\int_t^T|g(s,0,0)|^2ds\nonumber
\\&+\frac{2}{\theta}\int_t^T|f(s,0,0)|^2ds
+2\int_t^T e^{\beta s}Y_s dK_s \nonumber\\& +2\int_t^T e^{\beta
s}Y_s g(s,Y_s,Z_s)dB_s-2\int_t^T e^{\beta s}Y_s Z_s dW_s
\label{eq:3}
\end{align}
By a change of integration order, we obtain
\begin{align}
\int_t^T(\int_{-T}^0e^{\beta s}|Z_{s+u}|^2\gamma(du))ds&=\int_t^T(\int_{-T}^0e^{\beta (s+u)}e^{-\beta u}1_{s+u\geq0}|Z_{s+u}|^2\gamma(du))ds\nonumber\\&
=\int_{-T}^0\int_{(t+u)\vee 0}^{T+u}e^{\beta r}e^{-\beta u}1_{r\geq0}|Z_{r}|^2dr\gamma(du)
\nonumber\\&=\int_{0}^T\int_{r-t}^{(r-t)\wedge0}e^{\beta r}e^{-\beta u}1_{r\geq0}|Z_{r}|^2\gamma(du)dr\nonumber\\&
=\int_{0}^Te^{\beta r}|Z_{r}|^2(\int_{-T}^{0}e^{-\beta u}\gamma(du))dr\nonumber\\&
\leq \int_{0}^T \widetilde{\gamma}e^{\beta r}|Z_{r}|^2dr.
\end{align}
return back to (4)
\begin{align}
&e^{\beta t}|Y_t|^2+(\beta-\theta)\int_t^T  e^{\beta s}|Y_s|^2 ds+\int_t^T e^{\beta s}|Z_s|^2ds \nonumber\\&
\leq e^{\beta T}|\xi|^2+\widetilde{\gamma}(\frac{2L}{\theta}+2L)T\sup_{0\leq t\leq T}e^{\beta t}|Y_{t}|^2+\widetilde{\gamma}(\frac{2L}{\theta}+2\alpha)\int_0^Te^{\beta s}|Z_{s}|^2ds
\nonumber\\&+\frac{2}{\theta}\int_t^Te^{\beta s}|f(s,0,0)|^2ds+2\int_t^T e^{\beta s}Y_s dK_s \nonumber\\&
+2\int_t^T e^{\beta s}Y_s g(s,Y_s,Z_s)dB_s-2\int_t^T e^{\beta s}Y_s Z_s dW_s\nonumber\\&
+2\int_t^Te^{\beta s}|g(s,0,0)|^2ds.
\label{eq:3.3}
\end{align}
Firstly, we say that
\begin{align}
E\int_0^Te^{\beta s}|Z_{s}|^2ds\leq C E\int_t^T(e^{\beta s}|f(s,0,0)|^2ds+e^{\beta s}|g(s,0,0)|^2ds+\sup_{0\leq t\leq T}e^{\beta t}|Y_{t}|^2) \label{eq:3.4}
\end{align}
where $C>0$ depends on $\beta, \theta, \widetilde{\gamma}, \alpha, L$ and $T$. The estimate of (\ref{eq:3.4}) can be obtain as follows :
Since $D_1=\beta-\theta>0$ in (\ref{eq:3.3}) putting $t=0$ we  get
\begin{align*}
&|Y_0|^2+(1-\widetilde{\gamma}(\frac{2L}{\theta}+2\alpha))\int_0^T e^{\beta s}|Z_s|^2ds \nonumber\\&
\leq e^{\beta T}|\xi|^2+(\widetilde{\gamma}(\frac{2L}{\theta}+2L)T+\frac{1}{\epsilon})\sup_{0\leq t\leq T}e^{\beta t}|Y_{t}|^2
\nonumber\\&+\frac{2}{\theta}\int_0^Te^{\beta s}|f(s,0,0)|^2ds+\epsilon e^{\beta T}|K_T|^2 \nonumber\\&
+2\int_t^T e^{\beta s}Y_s g(s,Y_s,Z_s)dB_s-2\int_t^T e^{\beta s}Y_s Z_s dW_s\nonumber\\&
+2\int_0µ^Te^{\beta s}|g(s,0,0)|^2ds
\label{eq:3.3}
\end{align*}
where $\epsilon>0$, by (\ref{RBDSDE})
\begin{align*}
K_T=Y_0-\xi-\int_0^Tf(t,Y_t,Z_t)dt-\int_0^Tg(t,Y_t,Z_t)dB_t+\int_0^TZ_tdW_t.
\end{align*}
there exist a constant $\widetilde{C}$ depending on $L, \alpha$ and $T$ such that
\begin{align*}
e^{\beta T}|K_T|^2\leq  \widetilde{C}(|Y_0|^2+|\xi|^2+\int_0^T|f(t,0,0)|^2dt+\int_0^T|g(t,0,0)|^2dt  \\+\sup_{0\leq t\leq T}|Y_t|^2+\int_0^T|Z_t|^2dt+|\int_0^Tg(t,Y_t,Z_t)dB_t|^2+|\int_0^TZ_tdW_t|^2).
\end{align*}
we replace the last inequality in previous
\begin{align*}
&(1-\epsilon \widetilde{C})|Y_0|^2+(1-\widetilde{\gamma}(\frac{2L}{\theta}+2\alpha))\int_0^T e^{\beta s}|Z_s|^2ds+\epsilon \widetilde{C}\int_0^T |Z_s|^2ds \nonumber\\&
\leq (e^{\beta T}+\epsilon \widetilde{C})|\xi|^2+(\widetilde{\gamma}(\frac{2L}{\theta}+2L)T+\frac{1}{\epsilon})\sup_{0\leq t\leq T}e^{\beta t}|Y_{t}|^2+\epsilon \widetilde{C}\sup_{0\leq t\leq T}|Y_t|^2
\nonumber\\&+(\epsilon \widetilde{C}+\frac{2}{\theta})\int_0^Te^{\beta s}|f(s,0,0)|^2ds \nonumber\\&
+2|\int_0^T e^{\beta s}Y_s g(s,Y_s,Z_s)dB_s|+\epsilon \widetilde{C}|\int_0^T g(s,Y_s,Z_s)dB_s|^2\nonumber\\&+2|\int_0^T e^{\beta s}Y_s Z_s dW_s|+\epsilon \widetilde{C}|\int_0^T  Z_s dW_s|^2\nonumber\\&
+(2+\epsilon \widetilde{C})\int_0^Te^{\beta s}|g(s,0,0)|^2ds
\end{align*}
choosing $\epsilon$ small and $\theta$ such that $D_2=(1-\widetilde{\gamma}(\frac{2L}{\theta}+2\alpha))>0$, we obtain
\begin{align*}
&E\int_0^T e^{\beta s}|Z_s|^2ds \leq C_1E[\sup_{0\leq t\leq T}e^{\beta t}|Y_{t}|^2+\int_0^Te^{\beta s}|f(s,0,0)|^2ds+\int_0^Te^{\beta s}|g(s,0,0)|^2ds] \nonumber\\&
+C_1E|\int_0^T e^{\beta s}Y_s g(s,Y_s,Z_s)dB_s|+C_1E|\int_0^T e^{\beta s}Y_s Z_s dW_s|
\end{align*}
where $C_1$ depends on $\beta, \widetilde{\gamma}, \theta, \alpha, L$ and $T$. By the Burkholder Davis Gundy inequality, we have
\begin{align*}
E|\int_0^T e^{\beta s}Y_s Z_s dW_s|\leq &C_2E(\int_0^Te^{\beta s}|Y_s|^2|Z_s|^2ds)^{\frac{1}{2}}\\
&\leq C_2E(\sup_{0\leq t\leq T}e^{\beta t}|Y_t|^2)^{\frac{1}{2}}E(\int_0^Te^{\beta s}|Z_s|^2ds)^{\frac{1}{2}}\\
&\leq \frac{C_2^2}{\lambda}E(\sup_{0\leq t\leq T}e^{\beta t}|Y_t|^2)+\frac{\lambda}{2}E\int_0^Te^{\beta s}|Z_s|^2ds.
\end{align*}
and
\begin{align*}
E|\int_0^T e^{\beta s}Y_s g(s,Y_s,Z_s)dB_s|\leq &C_3E(\int_0^Te^{\beta s}|Y_s|^2|g(s,Y_s,Z_s)|^2ds)^{\frac{1}{2}}\\
&\leq C_3E(\sup_{0\leq t\leq T}e^{\beta t}|Y_t|^2)^{\frac{1}{2}}E(\int_0^Te^{\beta s}|g(s,Y_s,Z_s)|^2ds)^{\frac{1}{2}}\\
&\leq \frac{C_2^2}{\lambda}E(\sup_{0\leq t\leq T}e^{\beta t}|Y_t|^2)+\frac{\lambda}{2}E\int_0^Te^{\beta s}|g(s,Y_s,Z_s)|^2ds\\
&\leq \frac{C_2^2}{\lambda}E(\sup_{0\leq t\leq T}e^{\beta t}|Y_t|^2)+\frac{\lambda}{2}E\int_0^Te^{\beta s}(|g(s,0,0)|^2\\&+L|Y_s|^2+\alpha|Z_s|^2)ds.
\end{align*}
where $\lambda>0$. Finally, we return back to the last inequality, choosing $\lambda$ small and using the Fatou lemma, we obtain that there exist a constant $C>0$ depending on $\beta, \theta, \alpha,  \widetilde{\gamma}, L$ and $T$ such that
\begin{align*}
E\int_0^T e^{\beta s}|Z_s|^2ds \leq CE[\sup_{0\leq t\leq T}e^{\beta t}|Y_{t}|^2+\int_0^Te^{\beta s}|f(s,0,0)|^2ds+\int_0^Te^{\beta s}|g(s,0,0)|^2ds].
\end{align*}

In the second part of proof, we claim that
\begin{align}
E\sup_{0\leq t\leq T} e^{\beta t}|Y_t|^2 &\leq \widehat{C}E\bigg[ e^{\beta T}|\xi|^2+\int_0^Te^{\beta s}|f(s,0,0)|^2ds\nonumber\\&+\int_0^Te^{\beta s}|g(s,0,0)|^2ds+\sup_{0\leq t\leq T} e^{\beta t}(S_t^+)^2
\bigg].
\end{align}
holds for a positive constant $\widehat{C}$ depending on $\beta, \theta, \alpha,  \widetilde{\gamma}, L$ and $T$. To prove (8), going back to (6), using the fact $\int_0^T e^{\beta t}(Y_t-S_t)dK_t=0$, we get
\begin{align}
&e^{\beta t}|Y_t|^2+(\beta-\theta)\int_t^T  e^{\beta s}|Y_s|^2 ds+\int_t^T e^{\beta s}|Z_s|^2ds \nonumber\\&
\leq e^{\beta T}|\xi|^2+\widetilde{\gamma}(\frac{2L}{\theta}+2L)T\sup_{0\leq t\leq T}e^{\beta t}|Y_{t}|^2+\widetilde{\gamma}(\frac{2L}{\theta}+2\alpha)\int_0^Te^{\beta s}|Z_{s}|^2ds
\nonumber\\&+\frac{2}{\theta}\int_t^Te^{\beta s}|f(s,0,0)|^2ds+\frac{1}{\delta}\sup_{0\leq t\leq T}e^{\beta t}(S_t^+)^2+\delta(\int_t^T e^{\frac{\beta s}{2}}dK_s)^2\nonumber\\&
+2\int_t^T e^{\beta s}Y_s g(s,Y_s,Z_s)dB_s-2\int_t^T e^{\beta s}Y_s Z_s dW_s
+2\int_t^Te^{\beta s}|g(s,0,0)|^2ds
\label{eq:3.6}
\end{align}
where $\delta>0$ is a constant. Using \cite{bhmm} and (7), we have
\begin{align}
e^{\beta T}E[K_T-K_t]^2\leq KE[\sup_{0\leq t\leq T}e^{\beta t}|Y_{t}|^2+\int_0^Te^{\beta s}|f(s,0,0)|^2ds+\int_0^Te^{\beta s}|g(s,0,0)|^2ds
\label{eq:3.7}
\end{align}
a constant $K$ depending on $\beta, \theta, \alpha,
\widetilde{\gamma}, L$ and $T$. Now taking expectation in (9) and
taking into account of (\ref{eq:3.7}), we obtain
\begin{align*}
&(1-\widetilde{\gamma}(\frac{2L}{\theta}+2\alpha)E\int_t^T e^{\beta s}|Z_s|^2ds \nonumber\\&
\leq E e^{\beta T}|\xi|^2+\big(\widetilde{\gamma}(\frac{2L}{\theta}+2L)T+\delta K\big)E\big(\sup_{0\leq t\leq T}e^{\beta t}|Y_{t}|^2\big)\nonumber\\&+(\frac{2}{\theta}+\delta K)E\int_t^Te^{\beta s}|f(s,0,0)|^2ds+\frac{1}{\delta}E(\sup_{0\leq t\leq T}e^{\beta t}(S_t^+)^2)\nonumber\\&
+2\delta KE\int_t^Te^{\beta s}|g(s,0,0)|^2ds
\label{eq:3.6}
\end{align*}
thus
\begin{align}
&E\int_t^T e^{\beta s}|Z_s|^2ds
\leq D_{2}^{-1} E e^{\beta T}|\xi|^2+\big(\widetilde{\gamma}(\frac{2L}{\theta}+2L)T+\delta K\big)E\big(\sup_{0\leq t\leq T}e^{\beta t}|Y_{t}|^2\big)\nonumber\\&+D_{2}^{-1}(\frac{2}{\theta}+\delta K)E\int_t^Te^{\beta s}|f(s,0,0)|^2ds+D_{2}^{-1} \frac{1}{\delta}E(\sup_{0\leq t\leq T}e^{\beta t}(S_t^+)^2)\nonumber\\&
+2 D_{2}^{-1}\delta KE\int_t^Te^{\beta s}|g(s,0,0)|^2ds
\end{align}
Next going back to (9) and (10), we obtain
\begin{align}
&E\big[\sup_{0\leq t\leq T} e^{\beta t} |Y_t|^2\big] \nonumber\\&
\leq E e^{\beta
T}|\xi|^2+\big(\widetilde{\gamma}(\frac{2L}{\theta}+2L)T\big)E\big(\sup_{0\leq
t\leq T}e^{\beta
t}|Y_{t}|^2\big)\nonumber\\&+(\frac{2}{\theta})E\int_t^Te^{\beta
s}|f(s,0,0)|^2ds+\frac{1}{\delta}E(\sup_{0\leq t\leq T}e^{\beta
t}(S_t^+)^2)\nonumber\\& +2E\int_t^Te^{\beta
s}|g(s,0,0)|^2ds+(\widetilde{\gamma}(\frac{2L}{\theta}+2\alpha)
E\int_t^T e^{\beta s}|Z_s|^2ds\nonumber\\&+\delta(\int_t^T
e^{\frac{\beta s}{2}}dK_s)^2 +2E\sup_{0\leq t\leq T}\big|\int_t^T
e^{\beta s}Y_s g(s,Y_s,Z_s)dB_s +\int_t^T e^{\beta s}Y_s Z_s dW_s|
\label{eq:3.9}
\end{align}
Using the Burkholder Davis Gundy inequality
, there exist a constant
$C_4>0$ such that
\begin{align*}
&2E\sup_{0\leq t\leq T}\big|\int_t^T e^{\beta s}Y_s Z_s dW_s\big|\\&
\leq C_4E\big[\rho\sup_{0\leq t\leq T}e^{\beta t}|Y_{t}|^2+\frac{1}{\rho}\int_t^T e^{\beta s}|Z_s|^2ds\big]
\end{align*}
and
\begin{align*}
&2E\sup_{0\leq t\leq T}\big|\int_t^T e^{\beta s}Y_s g(s,Y_s,Z_s)dB_s\big|\\&
\leq C_4E\big[\rho\sup_{0\leq t\leq T}e^{\beta t}|Y_{t}|^2+\frac{1}{\rho}\int_t^T e^{\beta s}|g(s,Y_s,Z_s)|^2ds\big]\\&
\leq C_4E\big[\rho\sup_{0\leq t\leq T}e^{\beta t}|Y_{t}|^2+\frac{1}{\rho}\int_t^T e^{\beta s}(|g(s,0,0)|^2+L|Y_s|^2+\alpha|Z_s|^2)ds\big]
\end{align*}

where $\rho>0$. Plugging this inequality in (\ref{eq:3.9}), we get
\begin{align*}
&E\big[\sup_{0\leq t\leq T} e^{\beta t} |Y_t|^2\big] \nonumber\\&\leq
E e^{\beta T}|\xi|^2+\big(\widetilde{\gamma}(\frac{2L}{\theta}+2L)T+C_4(\rho+\frac{LT}{\rho})+C_4\rho+\delta C\big)E\big(\sup_{0\leq t\leq T}e^{\beta t}|Y_{t}|^2\big)\nonumber\\&+(\frac{2}{\theta}+\delta C)E\int_t^Te^{\beta s}|f(s,0,0)|^2ds+\frac{1}{\delta}E(\sup_{0\leq t\leq T}e^{\beta t}(S_t^+)^2)\nonumber\\&
+(2+\frac{C_4}{\rho}+2\delta C)E\int_t^Te^{\beta s}|g(s,0,0)|^2ds+(\widetilde{\gamma}(\frac{2L}{\theta}+2\alpha)+C_4(\frac{1}{\rho}+\alpha))E\int_t^T e^{\beta s}|Z_s|^2ds.
\end{align*}
Finally it is enough to choose $\rho=\frac{1}{2C_4}$ and $\delta$ small to obtain (8). From the above inequalities (\ref{eq:3.4}), (8) and (\ref{eq:3.7}), we obtain that there exists a positive constant $\widehat{C}$ depending on $\beta, \theta, \alpha,  \widetilde{\gamma}, L$ and $T$ such that
\begin{align*}
&E(\sup_{0\leq t\leq T}e^{\beta t}|Y_t|^2+\int_0^T e^{\beta t}|Z_t|^2dt+e^{\beta T}|K_T|^2)\\&\leq
\widehat{C}E(e^{\beta T}|\xi|^2+\int_0^T e^{\beta t}|f(t,0,0)|^2dt+\int_0^T e^{\beta t}|g(t,0,0)|^2dt+\sup_{0\leq t\leq T}e^{\beta t}(S_t^+)^2).
\end{align*}

\end{proof}

\begin{proposition}
Let $(\xi,f,g,S)$ and $(\xi',f',g,S')$ be two triplets satisfying the
above assumptions (H1)-(H5). Suppose that $(Y, Z, K)$ is a solution
of the RBDSDE with time delayed generator $(\xi,f,g,S)$ and
$(Y',Z',K')$ is a solution of the RBDSDE with time delayed generator
$(\xi',f',g,S')$. Define
\begin{align*}
&\Delta\xi=\xi-\xi', \ \ \Delta f = f- f', \ \ \Delta S = S -S';\\ &\Delta Y =
Y-Y',\ \  \Delta Z = Z-Z',\ \ \Delta K = K-K'.
\end{align*} If the Lipschitz constant $L$ of the generator $f$ and $g$ is small
enough, then there exist two positive constants $\beta$ and $\epsilon$
satisfying that
\begin{align*}
D_1 := \beta-\epsilon> 0, D_2 := 1-\widetilde{\gamma}(\frac{2L}{\epsilon}+2\alpha)> 0
\end{align*}
and a positive constant $D$
depending on $\beta,\epsilon ,\widetilde{\gamma}, L, \alpha$ and $T$ such that
\begin{align}
&E(sup_{0\leq t\leq T} e^{\beta t}|\Delta Y_t|^2 + \int^T_0 e^{\beta t}|\Delta Z(t)|^2dt\nonumber\\&
\leq D E\big(e^{\beta t}|\Delta\xi|^2+\int^T_0e^{\beta t}|\Delta f (t, Y_t , Z_t )|^2dt\big)\nonumber\\&
+ D\bigg[ E\big(\sup_{0\leq t\leq T}e^{\beta t}|\Delta S(t)|^2\big)\bigg]^{\frac{1}{2}}\psi^{\frac{1}{2}}_{T} ,\label{eq:3.10}
\end{align}
 where
\begin{align*}
&\psi= E\bigg[e^{\beta T}|\xi|^2 +\int^T_0e^{\beta t}|f (t, 0, 0)|^2dt+\int^T_0e^{\beta t}|g (t, 0, 0)|^2dt +\sup_{0\leq t\leq T}e^{\beta t}(S(t)^+)^2\\&
 + e^{\beta T}|\xi'|^2 +\int^T_0e^{\beta t}|f' (t, 0, 0)|^2dt +\sup_{0\leq t\leq T}e^{\beta t}(S(t)^{'+})^2\bigg].
\end{align*}
\end{proposition}
\begin{proof}
Applying the Itô formula to $e^{\beta t}|Y_t|^2$ yields that
\begin{align}
&e^{\beta t}|\Delta Y_t|^2+\int_t^T \beta e^{\beta s}|\Delta Y_s|^2 ds \nonumber\\&
\leq e^{\beta T}|\Delta\xi|^2+2\int_t^T e^{\beta s}(\Delta Y_s) \Delta f(s,Y_s,Z_s)ds+2\int_t^T e^{\beta s}\Delta Y_s d\Delta K_s \nonumber\\&
+2\int_t^T e^{\beta s}\Delta Y_s( g(s,Y_s,Z_s)-g(s,Y'_s,Z'_s))dB_s-2\int_t^T e^{\beta s}\Delta Y_s \Delta Z_s dW_s\nonumber\\&
+\int_t^T e^{\beta s}|g(s,Y_s,Z_s)-g(s,Y'_s,Z'_s)|^2ds-\int_t^T e^{\beta s}|\Delta Z_s|^2ds\nonumber\\&
+\epsilon\int_t^T e^{\beta s}|\Delta Y_s|^2ds+\frac{1}{\epsilon}\int_t^Te^{\beta s} |f'(s,Y_s,Z_s)-f'(s,Y'_s,Z'_s)|^2ds \label{eq:3.11}
\end{align}
where $\epsilon>0$, now we use the relation $\int_t^Te^{\beta s }(\Delta Y_s-\Delta S_s)d\Delta K_s\leq 0$ and (5) and taking expectation, we have
\begin{align*}
&(\beta-\epsilon)E\int_t^T  e^{\beta s}|\Delta Y_s|^2 ds+(1-\widetilde{\gamma}(\frac{2L}{\epsilon}+2\alpha))E\int_t^T e^{\beta s}|\Delta Z_s|^2ds \nonumber\\&
\leq Ee^{\beta T}|\Delta \xi|^2+\widetilde{\gamma}(\frac{2L}{\epsilon}+2L)TE(\sup_{0\leq t\leq T}e^{\beta t}|\Delta Y_{t}|^2)\nonumber\\&+2E\int_t^T e^{\beta s}|\Delta S_s| dK_s \nonumber\\&+2E\int_t^T e^{\beta s}(\Delta Y_s) \Delta f(s,Y_s,Z_s)ds
\end{align*}
with the Holder inequality
\begin{align*}
E\int_0^Te^{\beta s}|\Delta S_s|d\Delta K_s\leq \bigg[E\big(\sup_{0\leq t\leq T}e^{\beta t}|\Delta S_t|^2\big)\bigg]^{\frac{1}{2}}\bigg[E\big(e^{\beta T}|\Delta K_T|^2\big)\bigg]^{\frac{1}{2}}
\end{align*}
and since $E|\Delta K_T|^2\leq C\big(E|K_T|^2+E|K'_T|^2\big)$, using inequalities (\ref{eq:3.7}) and proposition 2, we deduce that
\begin{align*}
Ee^{\beta T}|\Delta K_T|^2\leq C_5\psi_T,
\end{align*}
where $C_5>0$ depends on $\beta, \epsilon, \widetilde{\gamma}, \alpha, L$ and $T$. Therefore, we obtain
\begin{align}
&E\int_t^T e^{\beta s}|\Delta Z_s|^2ds
\leq D_{2}^{-1} E e^{\beta T}|\Delta\xi|^2+2D_{2}^{-1}E\int_0^T e^{\beta s}|\Delta f(s,Y_s,Z_s)\Delta Y_s|ds\nonumber\\&
+\frac{T}{\frac{1}{\widetilde{\gamma}(\frac{2L}{\epsilon}+2L)}-1}E\big(\sup_{0\leq t\leq T}e^{\beta t}|\Delta Y_{t}|^2\big)+C_6\bigg[E(\sup_{0\leq t\leq T}e^{\beta t}(S_t^+)^2)\bigg]^{\frac{1}{2}}(\psi_T)^\frac{1}{2}\label{eq:3.12}
\end{align}
where $C_6>0$ depends on $\beta, \epsilon, \widetilde{\gamma}, \alpha, L$ and $T$. But
\begin{align*}
2\int_0^T e^{\beta s}|\Delta f(s,Y_s,Z_s)\Delta Y_s|ds&\leq 2\sup_{0\leq t\leq T}e^{\frac{\beta t}{2}}|\Delta Y_{t}|\int_0^Te^{\frac{\beta s}{2}}|\Delta f(s,Y_s,Z_s)|ds\\&\leq\varepsilon \sup_{0\leq t\leq T}e^{\beta t}|\Delta Y_{t}|^2+\frac{1}{\varepsilon}\int_0^Te^{\beta s}|\Delta f(s,Y_s,Z_s)|^2ds,
\end{align*}
where $\varepsilon>0$. With (\ref{eq:3.11}), the Burkholder Davis Gundy inequality and the two previous inequalities, we get
\begin{align}
&(1-\widetilde{\gamma}T(\frac{L}{\epsilon}+L))E\big[\sup_{0\leq t\leq T}e^{\beta t}|\Delta Y_t|^2\big]\nonumber\\&\leq
E\bigg[e^{\beta T}|\Delta \xi|^2+\varepsilon \sup_{0\leq t\leq T}e^{\beta t}|\Delta Y_{t}|^2+\frac{1}{\varepsilon}\int_0^Te^{\beta s}|\Delta f(s,Y_s,Z_s)|^2ds\bigg]\nonumber\\&+C_6\bigg[E(\sup_{0\leq t\leq T}e^{\beta t}(S_t^+)^2)\bigg]^{\frac{1}{2}}(\psi_T)^\frac{1}{2}\nonumber\\&+E\bigg[\widetilde{\gamma}(\frac{L}{\epsilon}
+\alpha)\int_0^Te^{\beta s}|\Delta Z_s|^2ds
+2c\rho\sup_{0\leq t\leq T}|\Delta Y_s|^2+\frac{c}{\rho}\int_t^Te^{\beta s}|\Delta Z_s|^2ds\nonumber\\&+\frac{c}{\rho}|\int_t^Te^{\beta s}(|g(s,0,0)|^2+L|\Delta Y_s|^2+\alpha|\Delta Z_s|^2)\bigg]\label{eq:3.13}
\end{align}
where $c,\ \rho$ are constants.

By (\ref{eq:3.12}) and (\ref{eq:3.13}), we have

\begin{align*}
&(1-\widetilde{\gamma}T(\frac{L}{\epsilon}+L)-(\frac{T}{\frac{1}{\widetilde{\gamma}
(\frac{2L}{\epsilon}+2L)-1}}+D_{2}^{-1}\varepsilon)
(\widetilde{\gamma}\frac{L}{\epsilon}+L+\frac{c}{\rho})-\varepsilon-c\rho)
E\big[\sup_{0\leq t\leq T}e^{\beta t}|\Delta Y_t|^2\big]\\&
\\&\leq
(1+D_{2}^{-1}(\widetilde{\gamma}(\frac{L}{\epsilon}
+\alpha)+\frac{c}{\rho}+\frac{\alpha c}{\rho}))E[e^{\beta T}|\Delta \xi|^2]\\&+[1+D_{2}^{-1}(\widetilde{\gamma}(\frac{L}{\epsilon}
+\alpha)+\frac{c}{\rho})]\frac{1}{\varepsilon}E\int_0^Te^{\beta s}|\Delta f(s,Y_s,Z_s)|^2ds\\&+C_6[1+\widetilde{\gamma}(\frac{L}{\epsilon}
+\alpha)+\frac{c}{\rho}+\frac{\alpha c}{\rho}]\big[E(\sup_{0\leq t\leq T}e^{\beta t}(S_t^+)^2)\big]^{\frac{1}{2}}(\psi_T)^\frac{1}{2}\\&+\frac{c}{\rho}E\int_t^Te^{\beta s}|g(s,0,0)|^2ds
\end{align*}

Finally, it is enough to choose $\rho=\frac{1}{2c}$ and
$\varepsilon$ small in the above inequalities (\ref{eq:3.12}) and
(\ref{eq:3.13}) to obtain (\ref{eq:3.10}). The proposition is proved
\end{proof}

We deduce immediately the following uniqueness result from proposition 3 with $\xi'=\xi,\ f=f'$ and $S=S'$.
\begin{corollary}
Under assumptions $H1)-H5)$ , if the Lipschitz constant L of the
generator f is small enough and for two positive constants â and ã
the conditions of Proposition 3 are satisfied, then there exists at
most one solution of the RBSDE with time delayed generators
$i)-iv).$
\end{corollary}
\section{Existence and uniqueness of the solution}
To begin with, let us first assume that $f$ does not depend on $(y,
z)$, that is, it is a given $\mathfrak{F}_t$-progressively
measurable process satisfying that

(H2' ) $E\int^T_0 |f (t)|^2dt <\infty$ and $E\int^T_0 |g (t)|^2dt <\infty$.

 A solution to the backward reflection problem is a triple $(Y, Z, K)$ which
satisfies (i), (iii), (iv) and

(ii') $Y(t) = \xi + \int^T_t f (s)ds + K(T)-K(t)+\int^T_tg(s)dB_s-\int^T_t Z(s)dW(s), 0\leq t\leq T .$

The following proposition is from Bahlali et al. (2009) .
\begin{proposition}
Under assumptions (H1), (H2' ) and (H5), the Backward Reflected Problem  (i), (ii'), (iii),
(iv) has a unique solution ${(Y_t, Z_t, K_t); 0\leq t\leq T }.$
\end{proposition}
We now deal with the general case of generator i.e. f depends on (y,
z).
\begin{theorem}
Assume assumptions H1)-H5) hold. If the Lipschitz constant $L$ of
the generator $f$ and $g$ is small enough and for two positive constants
$\beta$ and $\epsilon$ the conditions of Proposition 3 are satisfied, then
the RBDSDE with time delayed generator (i)-(iv) has a unique solution
$(Y_t, Z_t, K_t); 0\leq t\leq T.$
\end{theorem}

\begin{proof}
For any $\left(Y_t,Z_t,K_t\right); 0\leq t\leq T $ which satisfy (i)
and (iii). Given $(U,V)\in S^{2}\times M^{2}~$, let
\begin{eqnarray*}
&&Y_{t} =\xi+\int_{t}^{t}f\left( s,U_{s},V_{s}\right)ds+\int_{t}^{t}g\left( s,U_{s},V_{s}\right) dB_{s}-\int_{t}^{t}Z_{s}dW_{s}, \\
&&Y_t\geq S_t,\\&& \int_0^Te^{\beta t}(Y_t-S_t)dK_t=0,\ \ \texttt{for some}\ \beta>0.
\end{eqnarray*}%
the solution of this equation exists and is unique by proposition 5.
Hence, if we define $\Theta (U,V)=(Y,Z)$, then $\Theta $ maps
$S^{2}\times M^{2}~$ to itself. We show now that $\Theta $ is
contractive. To this end, take any $\left( U^{i},V^{i}\right) \in
S^{2}\times M^{2}$ and $\left( Y^{i},Z^{i}\right) \in S^{2}\times
M^{2},\ \ (i=1,2)$, and let $\Theta (U^{i},V^{i})=(Y^{i},Z^{i}).$ We
denote \ $\left( \overline{Y},\overline{Z}, \overline{K}\right)
=\left( Y^{1}-Y^{2},Z^{1}-Z^{2},K^{1}-K^{2}\right) $ and $\left(
\overline{U},\overline{V},\right) =\left(
U^{1}-U^{2},V^{1}-V^{2}\right) $. Therefore, It\^{o}'s formula
applied to $ \left\vert \overline{Y}\right\vert ^{2}e^{\beta t}$ and
the inequality $2ab\leq \left( \frac{1}{\epsilon }\right)
a^{2}+\epsilon b^{2}$, lead to
\begin{align*}
&e^{\beta t}|\overline{Y}_t|^2+\int_t^T \beta e^{\beta s}|\overline{Y}_s|^2 ds \nonumber\\&
\leq
\epsilon\int_t^T e^{\beta s}|\overline{Y}_s|^2ds+\widetilde{\gamma}(\frac{L}{\epsilon}+2L)T\sup_{0\leq t\leq T} e^{\beta t}|\overline{U}_s|^2+\widetilde{\gamma}(\frac{L}{\epsilon}+2\alpha)\int_t^Te^{\beta s}|\overline{V}_s|^2ds\\&+2\int_t^T e^{\beta s}\overline{Y}_s d\overline{K}_s
+2\int_t^T e^{\beta s}\overline{Y}_s [g(s,U_s,V_s)-g(s,U'_s,V'_s)]dB_s-2\int_t^T e^{\beta s}\overline{Y}_s \overline{Z}_s dW_s\nonumber\\&
-\int_t^T e^{\beta s}|\overline{Z}_s|^2ds
\end{align*}
where $\epsilon>0$, and from $2\int_t^T e^{\beta s}\overline{Y}_s d\overline{K}_s\leq0$, we obtain
\begin{align}
&e^{\beta t}|\overline{Y}_t|^2+\int_t^T \beta e^{\beta s}|\overline{Y}_s|^2 ds+\int_t^T e^{\beta s}|\overline{Z}_s|^2ds \nonumber \\&\leq
\epsilon\int_t^T e^{\beta s}|\overline{Y}_s|^2ds+\widetilde{\gamma}(\frac{L}{\epsilon}+2L)T\sup_{0\leq t\leq T} e^{\beta t}|\overline{U}_s|^2+\widetilde{\gamma}(\frac{L}{\epsilon}+2\alpha)\int_t^Te^{\beta s}|\overline{V}_s|^2ds \nonumber\\&
+2\int_t^T e^{\beta s}\overline{Y}_s [g(s,U_s,V_s)-g(s,U'_s,V'_s)]dB_s-2\int_t^T e^{\beta s}\overline{Y}_s \overline{Z}_s dW_s\label{eq3.14}
\end{align}
taking the expectation we have
\begin{align*}
&e^{\beta t}E|\overline{Y}_t|^2+(\beta-\epsilon)E\int_t^T e^{\beta s}|\overline{Y}_s|^2 ds+E\int_t^T e^{\beta s}|\overline{Z}_s|^2ds \\&\leq
\widetilde{\gamma}(\frac{L}{\epsilon}+2L)TE\sup_{0\leq t\leq T} e^{\beta t}|\overline{U}_s|^2+\widetilde{\gamma}(\frac{L}{\epsilon}+2\alpha)E\int_t^Te^{\beta s}|\overline{V}_s|^2ds
\end{align*}
By choosing $\beta-\epsilon>0$, we get
\begin{align}
E\int_t^T e^{\beta s}|\overline{Z}_s|^2ds \leq
\widetilde{\gamma}(\frac{L}{\epsilon}+2L+2\alpha)\max(T,1)E\bigg[\sup_{0\leq t\leq T} e^{\beta t}|\overline{U}_s|^2+\int_t^Te^{\beta s}E|\overline{V}_s|^2ds\bigg]\label{eq3.15}
\end{align}
Taking supremum in (\ref{eq3.14}), and the expectation, we get
\begin{align}
&E[\sup_{0\leq t\leq T}e^{\beta t}|\overline{Y}_t|^2] \nonumber \\&\leq
\widetilde{\gamma}(\frac{L}{\epsilon}+2L+2\alpha)\max(1,T)\bigg[\sup_{0\leq t\leq T} e^{\beta t}|\overline{U}_s|^2+\int_t^Te^{\beta s}|\overline{V}_s|^2ds \bigg]\nonumber\\&
+2E\big[\sup_{0\leq t\leq T}|\int_t^T e^{\beta s}\overline{Y}_s [g(s,U_s,V_s)-g(s,U'_s,V'_s)]dB_s+\int_t^T e^{\beta s}\overline{Y}_s \overline{Z}_s dW_s|\big]\label{eq3.16}
\end{align}
By the Burkholder Davis Gundy inequality, there exist real number
$\widehat{C}$, such that
\begin{align}
&E[\sup_{0\leq t\leq T}e^{\beta t}|\overline{Y}_t|^2] \nonumber
\\&\leq
\widetilde{\gamma}(\frac{L}{\epsilon}+2L+2\alpha)\max(1,T)\bigg[\sup_{0\leq
t\leq T} e^{\beta t}|\overline{U}_s|^2+\int_t^Te^{\beta
s}|\overline{V}_s|^2ds \bigg]\nonumber\\&
+\frac{1}{2}E\big[\sup_{0\leq t\leq T}| e^{\beta
t}|\overline{Y}_t|^2\bigg] +\widehat{C}E\int_0^T e^{\beta
s}|\overline{Z}_s|^2 ds\label{eq3.17}
\end{align}
Plugging now the inequality in (\ref{eq3.15}) and (\ref{eq3.17}), we obtain
\begin{align*}
&E[\sup_{0\leq t\leq T}e^{\beta t}|\overline{Y}_t|^2]+E\int_0^T
e^{\beta s} |\overline{Z}_s|^2 ds \nonumber \\&\leq
\widetilde{\gamma}(\frac{L}{\epsilon}+2L+2\alpha)\max(1,T)(2\widehat{C}+3)E\bigg[\sup_{0\leq
t\leq T} e^{\beta t}|\overline{U}_s|^2+\int_t^Te^{\beta
s}|\overline{V}_s|^2ds \bigg]
\end{align*}
with good choice of $\epsilon$, we deduce that
\begin{align*}
E[\sup_{0\leq t\leq T}e^{\beta t}|\overline{Y}_t|^2+\int_0^T e^{\beta s}|\overline{Z}_s|^2 ds] \nonumber \leq
\frac{1}{2}E\bigg[\sup_{0\leq t\leq T} e^{\beta t}|\overline{U}_s|^2+\int_0^Te^{\beta s}|\overline{V}_s|^2ds \bigg]
\end{align*}

Consequently the mapping $\Theta $ is a strict contraction on $S^{2}\times
M^{2}$ equipped with the norm
\begin{equation*}
\left\Vert \left( Y,Z\right) \right\Vert _{\beta }=\left(
E\int_{t}^{T}e^{\beta s}\left( \left\vert \overline{Y}_{s}\right\vert
^{2}+\left\vert \overline{Z}_{s}\right\vert ^{2}\right) ds\right) ^{\frac{1}{
2}}.
\end{equation*}%
Moreover, it has a unique fixed point, which is the unique solution of the
RBDSDE with delayed time generator $\left( \xi ,f,g,S\right) .$

\end{proof}

\end{document}